\documentclass[a4paper,11pt]{amsart}

\usepackage{epsfig}
\usepackage{amsthm,amsfonts}
\usepackage{amssymb,graphicx,color}
\usepackage[all]{xy}
\usepackage{verbatim}
\usepackage{hyperref}
\usepackage{enumitem}

\newtheorem{theorem}{Theorem}[section]
\newtheorem*{theorem*}{Theorem}
\newtheorem{lemma}[theorem]{Lemma}
\newtheorem{corollary}[theorem]{Corollary}
\newtheorem{proposition}[theorem]{Proposition}
\newtheorem{definition}[theorem]{Definition}

\newtheorem{example}[theorem]{Example}
\newtheorem{remark}[theorem]{Remark}

\newcommand{\R}{\mathbb{R}}
\newcommand{\C}{\mathbb{C}}

\newcommand{\N}{\mathbb{N}}
\newcommand{\din}{\mathrm{d}_{\mathrm{inn}}}
\newcommand{\dout}{\mathrm{d}_{\mathrm{out}}}

\newtheorem{innercustomthm}{Theorem}
\newenvironment{customthm}[1]
  {\renewcommand\theinnercustomthm{#1}\innercustomthm}
  {\endinnercustomthm}
  
\newtheorem{innercustomcor}{Corollary}
\newenvironment{customcor}[1]
  {\renewcommand\theinnercustomcor{#1}\innercustomcor}
  {\endinnercustomcor}

\title[Rigidity of algebraic sets with small distortion]{Rigidity of complex algebraic sets with small distortion}

\author[Alexandre Fernandes]{Alexandre Fernandes}
\address{Departamento de Matem\'atica, Universidade Federal do Cear\'a, Fortaleza, CE, Brazil}
\email{alex@mat.ufc.br}

\author[Edson Sampaio]{Edson Sampaio}
\address{Departamento de Matem\'atica, Universidade Federal do Cear\'a, Fortaleza, CE, Brazil}
\email{edsonsampaio@mat.ufc.br}

\date{\today}

\subjclass[2020]{14H50, 32S05, 14P10, 58K45}
\thanks{The first named author was partially supported by CNPq-Brazil grant  304700/2021-5 and 405618/2023-9. The second named author was partially supported by CNPq-Brazil grants 303375/2025-6, 404611/2025-7 and 405618/2023-9 and by the Serrapilheira Institute (grant number Serra -- R-2110-39576).
}
\keywords{Lipschitz geometry, Lipschitz normally embedded, algebraic curves, complex affine curves}

\begin{document}

\begin{abstract}
We prove that for positive integers $k$, $n$, and $d>1$, there is a maximal $c(k,n,d)>1$ such that if $X \subset \mathbb{C}^n$ is a connected $k$-dimensional complex algebraic set with degree at most $d$ and $dist(X) < c(k,n,d)$, then $X$ must be a complex affine linear subspace, where $dist(X)$ denotes the distortion of $X$ and is given by the supremum of the ratios $d_{inner}(p, q)/d_{outer}(p, q)$ for all $p,q\in X$ with $p\not= q$. We also prove that $c(1,n,d)=\sqrt{2}$ for any  positive integers $k$ and $n$. This result provides a sharp rigidity phenomenon for curves in the $c$-LNE category, showing that near-isometric behavior of the inner and outer metrics forces the curve to be as simple as possible.
\end{abstract}

\maketitle

\tableofcontents
\section{Introduction}\label{sec:intro}

The study of Lipschitz geometry of singular spaces has attracted considerable attention in recent years, particularly through the notion of \emph{Lipschitz Normal Embedding} (LNE). A seminal question in this area is to understand which algebraic or semialgebraic sets satisfy the property that their inner metric is bi-Lipschitz equivalent to their outer (ambient) metric. This notion was introduced by Birbrair and Mostowski~\cite{birbrair-mostowski} in the context of semialgebraic sets, where they established fundamental properties of normally embedded sets. The concept has since become a central object in the Lipschitz geometry of singularities, with deep connections to the topology and classification of singular spaces; see the survey by Fantini and Pichon~\cite{fantini-pichon} and the monograph~\cite{fernandes-sampaio-book} for a comprehensive account. Among the key developments, Kurdyka and Orro~\cite{kurdyka-orro} studied geodesic distances on subanalytic sets, providing essential tools for the metric analysis of singular spaces, and Fernandes and Sampaio~\cite{fernandes-sampaio-imrn} proved that tangent cones of LNE sets are themselves LNE (with an appendix by Pichon and Neumann).

A natural refinement of LNE is the concept of \emph{$c$-Lipschitz Normal Embedding}, where one controls the bi-Lipschitz constant $c \geq 1$ quantitatively. In the extremal case $c = 1$, the inner and outer metrics coincide, and the set must be convex (in the case of connected subsets of $\R^n$). As $c$ grows, more complicated geometries become possible. A fundamental question is:

\medskip
\noindent\textbf{Question.} \textit{For which values of $c$ does the $c$-LNE condition force strong geometric rigidity on algebraic sets?}
\medskip

In this paper, we show that such a constant $c>1$ exists for complex algebraic sets in general, and we answer this question for complex algebraic curves by presenting an explicitly constant in this case. 

Our first main result is the following:

\begin{customthm}{\ref*{thm:universal_maximal_distortion}}
For positive integers $k$, $n$, and $d$, there is a maximal $c(k,n,d)>1$ such that if $X \subset \C^n$ is a connected $k$-dimensional complex algebraic set with degree at most $d$ and $dist(X) < c(k,n,d)$, then $X$ must be a complex affine linear subspace.
\end{customthm}

When one restricts to the class of algebraic sets that have the tangent cone at infinity as a union of linear subspaces, we obtain an explicit and sharp constant.
\begin{customthm}{\ref*{thm:tang_cone_union_linear_spaces}}
Let $X \subset \C^n$ be a connected $d$-dimensional complex algebraic set such that its tangent cone at infinity, $C(X,\infty)$, is a union of complex linear subspaces. If $dist(X) < \sqrt{2}$, then $X$ must be a complex affine linear subspace.
\end{customthm}

Since the tangent cone at infinity of a complex algebraic curve is a union of complex lines, we obtain the following direct consequence. 
\begin{customcor}{\ref*{cor:algebraic_curves}}
Let $X \subset \C^n$ be a connected algebraic affine complex curve. If $dist(X) < \sqrt{2}$, then $X$ is a 
complex line.  
\end{customcor}

This result reveals a sharp rigidity phenomenon: any connected algebraic curve in~$\C^n$ that is $c$-LNE with $c < \sqrt{2}$ must be the simplest possible curve, namely an affine complex line. In particular, no irreducible curve of degree 
$\geq 2$ and no reducible curve can satisfy the $c$-LNE condition with such a 
small constant.

\subsection*{Strategy of the proof}

The proof relies on two main ingredients. First, Theorem~\ref{thm:c-lne-tangent-cone-infinity} (a version at infinity of \cite[Theorem~5.2]{fernandes-sampaio-book}) shows that the tangent cone at infinity $C(X, \infty)$ inherits the $c$-LNE constant. Since $C(X, \infty)$ is a complex algebraic cone of dimension one, it is a finite union of complex lines through the origin. The computation in Example~\ref{ex:two-lines-sharp} then shows that two distinct lines always have $c$-LNE constant $\geq \sqrt{2}$, so the hypothesis $c < \sqrt{2}$ forces $C(X, \infty)$ to be a single line. The Lipschitz rigidity result of \cite{fernandes-sampaio-jga} then implies that $X$ itself must be an affine linear subspace.



\section{Preliminaries}\label{sec:prelim}

\subsection{Inner and outer metrics}

Let $X \subset \C^n$ be a connected set. The \emph{outer metric} (or 
\emph{ambient metric}) on $X$ is the restriction of the Euclidean metric on 
$\C^n$:
\[
\dout(x, y) = \|x - y\|, \quad x, y \in X.
\]
The \emph{inner metric} (or \emph{length metric}) on $X$ is defined by
\[
\din(x, y) = \inf_\gamma \operatorname{length}(\gamma),
\]
where the infimum is taken over all rectifiable paths $\gamma \colon [0,1] \to X$ 
joining $x$ to $y$. If no such path exists, we set $\din(x, y) = +\infty$.

It is immediate that $\dout(x,y) \leq \din(x,y)$ for all $x, y \in X$.

\subsection{Lipschitz normal embedding}

\begin{definition}\label{def:c-lne}
Let $c \geq 1$. A connected subset $X \subset \C^n$ is said to be 
\emph{$c$-Lipschitz Normally Embedded} (or \emph{$c$-LNE}) if
\[
\din(x, y) \leq c \cdot \dout(x, y) \quad \text{for all } x, y \in X.
\]
When such a constant $c$ exists (without specifying its value), $X$ is called 
\emph{Lipschitz Normally Embedded} (LNE).
\end{definition}

\begin{remark}
The condition $c = 1$ means $\din = \dout$ on $X$, which, for connected compact subsets of Euclidean space, implies convexity. Thus, the $c$-LNE condition interpolates between convexity and general LNE.
\end{remark}

\begin{example}\label{ex:basic}
Every affine linear subspace of $\C^n$ is $1$-LNE (being convex). Every smooth compact submanifold of $\R^n$ is LNE, though the optimal constant may be strictly greater than~$1$.
\end{example}

\begin{example}\label{ex:two-lines-sharp}
	Let $L_1, L_2 \subset \C^n$ be two distinct complex lines through the 
	origin, and let $X = L_1 \cup L_2$. Since $L_1 \cap L_2 = \{0\}$, any 
	rectifiable path in $X$ connecting a point $x \in L_1 \setminus \{0\}$ to 
	a point $y \in L_2 \setminus \{0\}$ must pass through the origin. Hence
	\[
	\din(x, y) = \|x\| + \|y\|, \quad x \in L_1,\ y \in L_2.
	\]
	We now compute the sharp $c$-LNE constant of $X$. Choose unit vectors 
	$u_0 \in L_1$ and $v_0 \in L_2$, and write $x = \lambda u_0$, 
	$y = \mu v_0$ with $\lambda, \mu \in \C^*$. Then
	\[
	c^2 = \sup_{\lambda, \mu \in \C^*} 
	\frac{(|\lambda| + |\mu|)^2}
	{|\lambda|^2 + |\mu|^2 - 2\operatorname{Re}(\bar{\lambda}\mu\,
		\langle u_0, v_0 \rangle)}.
	\]
	Writing $\lambda = s e^{i\alpha}$ and $\mu = t e^{i\beta}$ with $s, t > 0$, 
	the denominator becomes $s^2 + t^2 - 2st\operatorname{Re}(e^{i(\beta - \alpha)}
	\langle u_0, v_0\rangle)$. This is minimized when 
	$\operatorname{Re}(e^{i(\beta - \alpha)}\langle u_0, v_0\rangle)$ is maximized, 
	which equals $|\langle u_0, v_0\rangle|$. Setting $r = t/s$, we thus need to 
	maximize
	\[
	f(r) = \frac{(1 + r)^2}{1 + r^2 - 2r\,|\langle u_0, v_0\rangle|}
	\]
	over $r > 0$. A direct computation gives
	\[
	f'(r) = \frac{2(1 + r)(1 - r)(1 + |\langle u_0, v_0\rangle|)}
	{(1 + r^2 - 2r\,|\langle u_0, v_0\rangle|)^2},
	\]
	which vanishes at $r = 1$ and is positive for $r < 1$, negative for $r > 1$. 
	The maximum is therefore achieved at $r = 1$ (i.e., $\|x\| = \|y\|$), giving
	\[
	c^2 = f(1) = \frac{4}{2 - 2\,|\langle u_0, v_0\rangle|} 
	= \frac{2}{1 - \cos\theta},
	\]
	where $\theta \in (0, \pi/2]$ is the Hermitian angle between $L_1$ and $L_2$, 
	defined by $\cos\theta = |\langle u_0, v_0\rangle|$.
	
	In particular, the smallest $c$-LNE constant among all pairs of distinct 
	complex lines is $c^2 = 2$, attained when $L_1 \perp L_2$. This shows that 
	the bound $c^2 < 2$ in Theorem~\ref{thm:tang_cone_union_linear_spaces} is sharp: the union of two 
	orthogonal complex lines is a connected algebraic affine complex curve that 
	is $\sqrt{2}$-LNE but is not a complex line.
\end{example}

\subsection{Tangent cones}\label{subsec:tangent-cones}

We recall the notion of the tangent cone of a subset of $\mathbb{K}^m$ (where 
$\mathbb{K}$ is $\R$ or $\C$) at infinity, following 
\cite{fernandes-sampaio-book}.

\begin{definition}\label{def:tangent-cone-infinity}
Let $X \subset \R^m$ be a non-empty set. We say that $v \in \R^m$ is a 
\emph{tangent vector of $X$ at infinity} if there exist a sequence of points 
$\{x_i\} \subset X$ tending to infinity and a sequence of positive real 
numbers $\{t_i\}$ such that
\[
\lim_{i \to \infty} \frac{1}{t_i}(x_i) = v.
\]
We denote by $C(X, \infty)$ the set of all tangent vectors of $X$ at infinity 
and call it the \emph{tangent cone of $X$ at infinity}.
\end{definition}

The tangent cone at infinity has the following fundamental property.

\begin{proposition}\label{prop:tangent-cone-infinity-algebraic}
Let $X \subset \C^n$ be a complex algebraic set. Then $C(X, \infty)$ is a 
complex algebraic set and $\dim_{\C} C(X, \infty) = \dim_{\C} X$. Moreover, $C(X, \infty)$ is a complex cone in the sense that: $x\in C(X, \infty)$ implies that the complex line through $0$ and $x$ also belongs to $C(X, \infty)$.
\end{proposition}

For a proof of Propositions~\ref{prop:tangent-cone-infinity-algebraic}, 
we refer the reader to \cite{fernandes-sampaio-book}.

\subsection{LNE and tangent cones at infinity}\label{subsec:lne-tangent-infinity}

The following result is the version at infinity of 
\cite[Theorem~5.2]{fernandes-sampaio-book}. The proof follows the 
same strategy.

\begin{theorem}\label{thm:c-lne-tangent-cone-infinity}
Let $X$ be a subset of $\R^m$ that is definable in some o-minimal structure on $\R$. If $X$ is $c$-Lipschitz normally 
embedded, then $C(X, \infty)$ is $c$-Lipschitz normally embedded.
\end{theorem}

\begin{proof}
Let $c \geq 1$ be such that
\[
\din(x, y) \leq c\,\|x - y\| \quad \text{for all } x, y \in X.
\]

Let $v, w \in C(X, \infty)$. We show that 
$d_{C(X,\infty)}(v, w) \leq c\,\|v - w\|$.

Let $\varepsilon > 0$ and set $K > c + 1 + \varepsilon$. We claim that 
for any $x, y \in X$, any rectifiable arc $\alpha$ on $X$ connecting $x$ 
to $y$ and satisfying
\[
\operatorname{Length}(\alpha) \leq \din(x, y) + \varepsilon\,\|x - y\|
\]
is contained in the closed ball 
$\overline{B}(0, K \cdot \max\{\|x\|, \|y\|\})$. Indeed, if $\alpha$ exits 
this ball, it must reach a point at distance at least 
$K\max\{\|x\|, \|y\|\}$ from the origin, and then return. The total length 
of such an arc is at least
\[
\bigl(K\max\{\|x\|,\|y\|\} - \|x\|\bigr) 
+ \bigl(K\max\{\|x\|,\|y\|\} - \|y\|\bigr) 
\geq 2(K - 1)\max\{\|x\|, \|y\|\}.
\]
On the other hand, the near-optimality condition and the triangle inequality 
$\|x - y\| \leq \|x\| + \|y\| \leq 2\max\{\|x\|, \|y\|\}$ give
\[
\operatorname{Length}(\alpha) 
\leq (c + \varepsilon)\,\|x - y\|
\leq 2(c + \varepsilon)\max\{\|x\|, \|y\|\}.
\]
Since $K > c + 1 + \varepsilon$, we have $2(K - 1) > 2(c + \varepsilon)$, 
yielding a contradiction. This proves the claim.

By definition of the tangent cone at infinity, there exist sequences 
$\{x_n\}$ and $\{y_n\}$ in $X$ and a sequence of positive real numbers 
$\{t_n\}$ with $t_n \to +\infty$ such that
\[
\frac{x_n}{t_n} \to v \quad \text{and} \quad \frac{y_n}{t_n} \to w.
\]
For each $n$, let $\gamma_n \colon [0, 1] \to X$ be a curve connecting 
$x_n$ to $y_n$ satisfying
\[
\operatorname{Length}(\gamma_n) \leq \din(x_n, y_n) 
+ \varepsilon\,\|x_n - y_n\|.
\]
Set $M_n := \max\{\|x_n\|, \|y_n\|\}$. By the claim above, $\gamma_n$ is 
contained in $\overline{B}(0, KM_n)$. Define 
$\alpha_n \colon [0, 1] \to \R^m$ by $\alpha_n = \frac{1}{t_n}\gamma_n$. 
Then $\alpha_n$ is contained in $\overline{B}(0, KM_n/t_n)$. Since 
$M_n/t_n \to \max\{\|v\|, \|w\|\}$, the curves $\alpha_n$ are eventually 
contained in a fixed bounded set. Moreover,
\[
\ell_n := \operatorname{Length}(\alpha_n) 
= \frac{1}{t_n}\operatorname{Length}(\gamma_n)
\leq \frac{c + \varepsilon}{t_n}\|x_n - y_n\|
= (c + \varepsilon)\left\|\frac{x_n}{t_n} - \frac{y_n}{t_n}\right\|.
\]
Since $\frac{1}{t_n}\|x_n - y_n\| \to \|v - w\|$, the sequence 
$\{\ell_n\}$ is bounded. By extracting a subsequence if necessary, we may 
assume $\ell_n \leq \ell := (c + \varepsilon)\|v - w\| + 1$ for all $n$.

For each $n$, let $\widetilde{\alpha}_n \colon [0, \ell_n] \to \R^m$ be the 
arc-length reparametrization of $\alpha_n$, and define 
$\beta_n \colon [0, \ell] \to \R^m$ by
\[
\beta_n(t) = 
\begin{cases}
\widetilde{\alpha}_n(t), & \text{if } t \in [0, \ell_n], \\
\widetilde{\alpha}_n(\ell_n), & \text{if } t \in [\ell_n, \ell].
\end{cases}
\]
Since $\widetilde{\alpha}_n$ is parametrized by arc length, it holds for 
all $n$ that
\[
\|\beta_n(s) - \beta_n(t)\| \leq |s - t| \quad \text{for all } s, t \in [0, \ell].
\]
Thus the family $\{\beta_n\}$ is equicontinuous and uniformly bounded. By 
the Arzel\`{a}--Ascoli theorem, there exists a continuous arc 
$\alpha \colon [0, \ell] \to C(X, \infty)$ such that, up to a subsequence, 
$\beta_{n_k} \rightrightarrows \alpha$.

Let us estimate the length of $\alpha$. Since $\varepsilon > 0$, there 
exists a positive integer $k_0$ such that, for all $k > k_0$,
\[
\operatorname{Length}(\beta_{n_k}) 
\leq (c + \varepsilon)\,\|v - w\| + \varepsilon,
\]
hence $\operatorname{Length}(\alpha) \leq (c + \varepsilon)\,\|v - w\| 
+ \varepsilon$, and therefore 
$d_{C(X,\infty)}(v, w) \leq (c + \varepsilon)\,\|v - w\| + \varepsilon$. 
Since $\varepsilon > 0$ was arbitrary, we conclude that
\[
d_{C(X,\infty)}(v, w) \leq c\,\|v - w\|. \qedhere
\]
\end{proof}

\section{The main results}\label{sec:main-result}

Before stating our main results, we present a geometric rigidity property of complex algebraic sets in the extremal case $c = 1$, which serves as both motivation and a key ingredient.

\begin{proposition}\label{prop:1-lne-rigidity}
Let $X \subset \C^m$ be a complex algebraic set. If $dist(X)=1$, that is, $\din(x, y) = \|x - y\|$ for all $x, y \in X$, then $X$ is an affine linear subspace of $\C^m$.
\end{proposition}

\begin{proof}
The proof proceeds in two steps.

\medskip
\noindent\textbf{Step~1: $X$ is convex.}
Let $x, y \in X$. Since $\din(x, y) = \|x - y\|$, for each $n \in \N$, there 
exists a rectifiable path $\gamma_n \colon [0, 1] \to X$ connecting $x$ to $y$ 
with
\[
\operatorname{Length}(\gamma_n) 
\leq \|x - y\| + \frac{1}{n}.
\]
Recall that for any rectifiable path $\gamma$ in $\R^{2m}$ joining $x$ to $y$,
\[
\operatorname{Length}(\gamma) 
\geq  \|x - y\|,
\]
with equality if and only if $\gamma$ is a reparametrization of the 
straight-line segment $[x, y]$. Let $\widetilde{\gamma}_n$ denote the 
arc-length reparametrization of $\gamma_n$. Since the lengths converge 
to $\|x - y\|$ and the curves are uniformly bounded, by the 
Arzel\`{a}--Ascoli theorem, a subsequence converges uniformly to a path 
$\gamma \colon [0, \|x-y\|] \to X$ of length exactly $\|x - y\|$. Since 
$X$ is closed and $\gamma$ achieves the minimal possible length, $\gamma$ 
must parametrize the segment $[x, y]$. In particular, $[x, y] \subset X$. 
Since $x, y \in X$ were arbitrary, $X$ is convex.

\medskip
\noindent\textbf{Step~2: $X$ is an affine linear subspace.}
Since $X$ is convex, for any $x, y \in X$, the real segment $[x, y]$ is 
contained in~$X$. We claim that the entire complex affine line through $x$ 
and $y$ is also contained in $X$.

Indeed, $X$ is defined by polynomial equations: $X = V(f_1, \ldots, f_k)$ for 
some polynomials $f_j \colon \C^m \to \C$. Fix $x, y \in X$ with $x \neq y$ 
and consider the complex line $L = \{x + z(y - x) \mathrel{;} z \in \C\}$. 
The restriction $g_j(z) := f_j(x + z(y - x))$ is a polynomial in 
$z \in \C$. Since $[x, y] \subset X$, we have $g_j(t) = 0$ for all 
$t \in [0, 1] \subset \R$. But a polynomial in one complex variable that 
vanishes on an infinite subset of $\C$ (here, the uncountable set $[0,1]$) 
must be identically zero. Therefore $g_j \equiv 0$, which means $L \subset 
V(f_j)$. Since this holds for all $j$, we conclude $L \subset X$.

Hence, for any two points $x, y \in X$, the complex affine line through $x$ 
and $y$ lies entirely in $X$. Fixing a point $p \in X$, the set 
$X - p := \{x - p \mathrel{;} x \in X\}$ is a complex algebraic set 
containing the origin that is closed under complex linear combinations: if 
$u, w \in X - p$ and $\lambda \in \C$, then the complex line through $u$ 
and $w$ in $X - p$ gives $u + \lambda(w - u) \in X - p$. It follows that 
$X - p$ is a complex linear subspace of $\C^m$, and therefore 
$X = p + (X - p)$ is a complex affine linear subspace.
\end{proof}

\subsection{Rigidity for small distortion for sets in general}

\begin{theorem}\label{thm:universal_maximal_distortion}
For positive integers $k$, $n$, and $d$, there is a maximal $c(k,n,d)>1$ such that if $X \subset \C^n$ is a connected $k$-dimensional complex algebraic set with degree at most $d$ and $dist(X) < c(k,n,d)$, then $X$ must be a complex affine linear subspace.
\end{theorem}

The proof of Theorem~\ref{thm:universal_maximal_distortion} reduces to the following statement about cones. In what follows, for $x\in\C^n$ and a non-empty set $A\subset\C^n$, we write $\mathrm{d}(x,A)=\inf_{a\in A}\|x-a\|$ for the Euclidean distance from $x$ to $A$; this should not be confused with the distortion $dist(\cdot)$.

\begin{proposition}\label{prop:cones}
Let $k\le n$ and $d$ be positive integers. There exists $c>1$ such that every pure $k$-dimensional complex algebraic cone $C\subset\C^n$ with $\deg C\le d$ and $dist(C)<c$ is a complex linear subspace.
\end{proposition}

\begin{proof}[Proof of Theorem~\ref{thm:universal_maximal_distortion}, assuming Proposition~\ref{prop:cones}]
Let $c>1$ be given by Proposition~\ref{prop:cones} and let $X\subset\C^n$ be a connected pure $k$-dimensional complex algebraic set with $\deg X\le d$ and $dist(X)<c$. In particular, $X$ is LNE.
Let $\overline{X}\subset\mathbb{P}^n$ be the projective closure of $X$ and $H_\infty$ the hyperplane at infinity. The tangent cone $C(X,\infty)$ is the affine cone over $\overline{X}\cap H_\infty$; since every irreducible component of $\overline{X}$ meets $H_\infty$ in a set of pure dimension $k-1$, $C(X,\infty)$ is a pure $k$-dimensional complex algebraic cone (see also Proposition~\ref{prop:tangent-cone-infinity-algebraic}), and
$\deg C(X,\infty)=\deg(\overline{X}\cap H_\infty)\le\deg\overline{X}=\deg X\le d$.
By Theorem~\ref{thm:c-lne-tangent-cone-infinity}, $dist(C(X,\infty))\le dist(X)<c$. By Proposition~\ref{prop:cones}, $C(X,\infty)$ is a linear subspace, and then \cite[Theorem~1.2]{fernandes-sampaio-jga} implies that $X$ is an affine linear subspace.

Finally, let $c(k,n,d)$ be the supremum of all $c>1$ with the property stated in Theorem~\ref{thm:universal_maximal_distortion}; by the above, this set is nonempty, so $c(k,n,d)>1$. If $dist(X)<c(k,n,d)$, then $dist(X)<c$ for some admissible $c$, so $X$ is an affine linear subspace. Hence $c(k,n,d)$ itself has the property, and it is maximal.
\end{proof}

\begin{proof}[Proof of Proposition~\ref{prop:cones}]
Suppose the statement is false. Then there is a sequence $\{X_j\}$ of pure $k$-dimensional complex algebraic cones in $\C^n$, none of which is a linear subspace, with $\deg X_j\le d$ and $dist(X_j)\to 1$.
A reduced pure $k$-dimensional cone of degree $1$ is a linear subspace, hence $\deg X_j\ge 2$; passing to a subsequence, we may assume $\deg X_j=e$ for all $j$, with $2\le e\le d$. Note that each $X_j$ is connected, since it is a cone.

\medskip
\noindent\textbf{Step~1: Compactness.}
Let $\overline{B}$ be the closed unit ball of $\C^n$, $\overline{B}_R=R\overline{B}$, and $K_j=X_j\cap\overline{B}$. By the Blaschke selection theorem, after passing to a subsequence, $K_j\to K$ in the Hausdorff distance, for some compact set $K\subset\overline{B}$.
Each $K_j$ contains $0$ and satisfies $\lambda K_j\subset K_j$ for all $\lambda\in\C$ with $|\lambda|\le1$; these properties pass to the limit. Set
\[
  Y=\bigcup_{t>0} tK .
\]
Then $Y\cap\overline{B}_R=RK$ for every $R>0$, so $Y$ is closed, $\lambda Y\subset Y$ for all $\lambda\in\C$, and
\[
  X_j\cap\overline{B}_R=RK_j\longrightarrow RK=Y\cap\overline{B}_R
  \quad\text{in the Hausdorff distance, for every } R>0.
\]
Let $\varepsilon_j=\sup_{x\in K_j}\mathrm{d}(x,Y)$. Then $\varepsilon_j\to0$ and, since $X_j$ and $Y$ are cones,
\begin{equation}\label{eq:close-to-Y}
  \mathrm{d}(x,Y)\le\varepsilon_j\,\|x\| \quad\text{for all } x\in X_j .
\end{equation}

\medskip
\noindent\textbf{Step~2: The limit has distortion one.}
We claim that for all $v,w\in Y$ there is a rectifiable path in $Y$ from $v$ to $w$ of length at most $\|v-w\|$.

Let $v,w\in Y$ and fix $R>\max\{\|v\|,\|w\|\}+2\|v-w\|+2$. By Step~1, there are $x_j,y_j\in X_j$ with $x_j\to v$ and $y_j\to w$. Let $c_j=dist(X_j)$ and choose rectifiable paths $\gamma_j$ in $X_j$ from $x_j$ to $y_j$ with
\[
  \ell_j:=\operatorname{Length}(\gamma_j)\le \Big(c_j+\tfrac1j\Big)\|x_j-y_j\| .
\]
Then $\limsup_j\ell_j\le\|v-w\|$; in particular, for $j$ large, $\ell_j\le\ell:=\|v-w\|+1$ and $\gamma_j\subset\overline{B}(x_j,\ell_j)\subset\overline{B}_R$.
Let $\beta_j\colon[0,\ell]\to\overline{B}_R$ be the arc-length parametrization of $\gamma_j$ on $[0,\ell_j]$, extended as the constant $y_j$ on $[\ell_j,\ell]$. The maps $\beta_j$ are $1$-Lipschitz and uniformly bounded, so by the Arzel\`{a}--Ascoli theorem a subsequence converges uniformly to a $1$-Lipschitz path $\beta\colon[0,\ell]\to\C^n$ with $\beta(0)=v$ and $\beta(\ell)=w$.
Since $\beta_j(t)\in X_j\cap\overline{B}_R$ and $X_j\cap\overline{B}_R\to Y\cap\overline{B}_R$, we get $\beta(t)\in Y$ for all $t$. By lower semicontinuity of length under uniform convergence,
\[
  \operatorname{Length}(\beta)\le\liminf_j \operatorname{Length}(\beta_j)
  =\liminf_j\ell_j\le\|v-w\| .
\]

\medskip
\noindent\textbf{Step~3: The limit is a $k$-dimensional linear subspace.}

\emph{(a) $Y$ is a complex linear subspace.}
A path from $v$ to $w$ of length at most $\|v-w\|$ is a parametrization of the segment $[v,w]$. Hence, by Step~2, $[v,w]\subset Y$ for all $v,w\in Y$, i.e., $Y$ is convex (compare with Step~1 in the proof of Proposition~\ref{prop:1-lne-rigidity}).
For $u,w\in Y$ and $\lambda\in\C$, we then have $u+w=2\cdot\tfrac{u+w}{2}\in Y$ and $\lambda u\in Y$. Thus $Y$ is a complex linear subspace of $\C^n$.

\emph{(b) $\dim_{\C}Y\ge k$.}
Let $V\subset\C^n$ be a linear subspace with $\dim V=n-k+1$. Since $\dim\mathbb{P}(X_j)+\dim\mathbb{P}(V)=(k-1)+(n-k)=n-1$, the projective dimension theorem gives $\mathbb{P}(X_j)\cap\mathbb{P}(V)\ne\emptyset$, so there is $u_j\in X_j\cap V$ with $\|u_j\|=1$. Any limit point of $\{u_j\}$ is a unit vector in $Y\cap V$. Thus $Y$ meets every linear subspace of dimension $n-k+1$ nontrivially, which forces $\dim Y\ge k$.

\emph{(c) $\dim_{\C}Y\le k$.}
Suppose that $\dim Y=m>k$, and let $\rho\colon\C^n\to Y$ be the orthogonal projection. Then $\rho|_{X_j}$ is a proper mapping, and thus $Z_j=\rho(X_j)$ is an algebraic cone in $Y\cong\C^{m}$ with $\dim Z_j\le k$ and $\deg Z_j\le\deg X_j=e$. Since an algebraic set of degree at most $e$ is the common zero set of polynomials of degree at most $e$, and $Z_j\ne Y$, there is a nonzero polynomial of degree at most $e$ vanishing on $Z_j$; as $Z_j$ is a cone, each of its homogeneous components vanishes on $Z_j$. So there is a nonzero homogeneous polynomial $h_j$ on $Y$ with $\deg h_j\le e$ and $h_j|_{Z_j}=0$.
Passing to a subsequence, $\deg h_j=m$ is constant and, after normalizing the coefficient vectors to norm one, $h_j\to h$ with $h\ne0$ homogeneous of degree $m$.
Let $w\in W$ with $\|w\|=1$. Since $w\in Y$, there are $x_j\in X_j$ with $x_j\to w$; then $\rho(x_j)\to\rho(w)=w$ and $h_j(\rho(x_j))=0$, so $h(w)=0$. Thus $h$ vanishes on the unit sphere of $Y$, hence on $Y$, a contradiction.

Therefore, $Y$ is a complex linear subspace of dimension $k$.

\medskip
\noindent\textbf{Step~4: Several sheets close to $Y$ give large distortion.}
After a unitary change of coordinates (which does not change distortions), we may assume $Y=\C^k\times\{0\}$. Write $x=(x',x'')\in\C^k\times\C^{n-k}$ and let $\pi(x)=x'$. By \eqref{eq:close-to-Y}, $\|x''\|\le\varepsilon_j\|x\|$ for $x\in X_j$; we take $j$ large so that $\varepsilon_j<1/2$. Then
\begin{equation}\label{eq:norm-bound}
  \|x\|\le\frac{\|x'\|}{\sqrt{1-\varepsilon_j^2}}
  \quad\text{and}\quad
  \|x''\|\le\frac{\varepsilon_j\|x'\|}{\sqrt{1-\varepsilon_j^2}}
  \qquad\text{for all } x\in X_j .
\end{equation}
In particular, $X_j\cap\ker\pi=\{0\}$ and $\pi|_{X_j}\colon X_j\to\C^k$ is proper with finite fibres. Moreover, for $a\in\C^k$, the fibre $X_j\cap\pi^{-1}(a)$ is $X_j\cap(a+\ker\pi)$; the projective closure of the affine plane $a+\ker\pi$ meets $H_\infty$ in $\mathbb{P}(\ker\pi)$, which is disjoint from $\overline{X_j}\cap H_\infty=\mathbb{P}(X_j)$. By B\'ezout's theorem, for generic $a$ the fibre consists of exactly $e$ points. Hence $\pi|_{X_j}$ is an analytic covering of $\C^k$ with $e\ge2$ sheets (see \cite{chirka}).

\emph{A ball free of branching, uniform in $j$.}
Choose a linear form $\ell_j$ on $\C^n$ separating the $e$ points of some generic fibre of $\pi|_{X_j}$ and, for $a\in\C^k$, let $x_1(a),\dots,x_e(a)$ be the points of $X_j\cap\pi^{-1}(a)$, listed with multiplicity. The elementary symmetric functions $\sigma_{j,i}(a)$ of $\ell_j(x_1(a)),\dots,\ell_j(x_e(a))$ are holomorphic on $\C^k$ (see \cite{chirka}). Since $X_j$ is a cone, the fibre over $ta$ is $t$ times the fibre over $a$ (with the same multiplicities), so $\sigma_{j,i}(ta)=t^i\sigma_{j,i}(a)$ for $t\in\C^*$. Hence $\sigma_{j,i}$ is a homogeneous polynomial of degree $i$, and
\[
  \Delta_j(a)=\prod_{1\le r<s\le e}\big(\ell_j(x_r(a))-\ell_j(x_s(a))\big)^2
\]
is a homogeneous polynomial of degree $e(e-1)$ on $\C^k$, not identically zero by the choice of $\ell_j$. If $\Delta_j(a)\ne0$, then the fibre over $a$ consists of $e$ distinct points.

Normalize $\Delta_j$ so that its coefficient vector has norm one (this does not change its zero set). Passing to a subsequence, $\Delta_j\to\Delta$ uniformly on compact sets, with $\Delta\ne0$ homogeneous. Choose $p\in\C^k$ with $\|p\|=1$ and $\Delta(p)\ne0$, and $\delta\in(0,1/2)$ such that $|\Delta|\ge2\mu>0$ on $\overline{B}(p,\delta)$. Then $|\Delta_j|\ge\mu$ on $\overline{B}(p,\delta)$ for all large $j$.

Fix such a $j$ and let $U=B(p,\delta)\subset\C^k$. Over $U$ every fibre of $\pi|_{X_j}$ has exactly $e$ points, each of multiplicity one, so $\pi\colon X_j\cap\pi^{-1}(U)\to U$ is a proper local homeomorphism with fibres of constant cardinality $e$, i.e., an $e$-sheeted covering map. Since $U$ is simply connected, the covering is trivial:
\[
  X_j\cap\pi^{-1}(U)=\Gamma_1\sqcup\dots\sqcup\Gamma_e ,
  \qquad \Gamma_i=\{(a,\varphi_i(a)) \mathrel{;} a\in U\},
\]
where the maps $\varphi_i\colon U\to\C^{n-k}$ are continuous and the sets $\Gamma_i$ are pairwise disjoint, each open and closed in $X_j\cap\pi^{-1}(U)$.

\emph{The estimate.}
Let $x=(p,\varphi_1(p))$ and $y=(p,\varphi_2(p))$. Then $x\ne y$ and, by \eqref{eq:norm-bound},
\[
  \|x-y\|\le\|\varphi_1(p)\|+\|\varphi_2(p)\|
  \le\frac{2\varepsilon_j}{\sqrt{1-\varepsilon_j^2}} .
\]
Let $\gamma\colon[0,1]\to X_j$ be any rectifiable path from $x$ to $y$. If $\pi(\gamma([0,1]))\subset U$, then $\gamma([0,1])$ would be a connected subset of $\Gamma_1\sqcup\dots\sqcup\Gamma_e$ meeting both $\Gamma_1$ and $\Gamma_2$, which is impossible. Hence there is $t_0$ with $\|\pi(\gamma(t_0))-p\|\ge\delta$. Since $\pi\circ\gamma$ starts and ends at $p$ and $\pi$ is $1$-Lipschitz,
\[
  \operatorname{Length}(\gamma)\ge\operatorname{Length}(\pi\circ\gamma)
  \ge 2\,\|\pi(\gamma(t_0))-p\|\ge2\delta .
\]
Therefore $\din(x,y)\ge2\delta$ and
\[
  dist(X_j)\ge\frac{\din(x,y)}{\|x-y\|}
  \ge\frac{\delta\sqrt{1-\varepsilon_j^2}}{\varepsilon_j}
  \xrightarrow[j\to\infty]{}+\infty ,
\]
which contradicts $dist(X_j)\to1$. This proves Proposition~\ref{prop:cones}.
\end{proof}

\begin{remark}
Steps~1--3 show that a cone of bounded degree with distortion close to $1$ is Hausdorff-close (on the unit ball) to a $k$-dimensional linear subspace, and Step~4 shows that a non-linear cone of bounded degree which is Hausdorff-close to a $k$-dimensional linear subspace has large distortion. The degree bound is used in the reduction to cones, in Step~3(c), and in the uniform choice of $\delta$ in Step~4. The argument is not effective; explicit values of $c(k,n,d)$ would require a quantitative version of Step~4.
\end{remark}

\begin{remark}
We assumed $X$ to be pure-dimensional. For non-pure sets, lower-dimensional components could connect different sheets in Step~4. This can be handled by also removing from $U$ the projections of the lower-dimensional components (which lie in proper hypersurfaces of bounded degree, so the same normalization argument applies), and by treating separately the case in which the $k$-dimensional part is a single $k$-dimensional linear subspace.
\end{remark}

\subsection{Rigidity of sets with tangent cones at infinity being union of linear subspaces}

Proposition~\ref{prop:1-lne-rigidity} reveals a striking rigidity phenomenon 
in complex algebraic geometry: the condition $\din = \dout$ alone forces a 
complex algebraic set to be flat. This raises a natural question: {\it how much can the constant $c$ be relaxed while still retaining geometric rigidity?}

In full generality, the answer depends on the dimension, but for curves, we obtain a sharp threshold.

\begin{theorem}\label{thm:tang_cone_union_linear_spaces}
Let $X \subset \C^n$ be a connected $d$-dimensional complex algebraic set such that its tangent cone at infinity, $C(X,\infty)$, is a union of complex linear subspaces. If $dist(X) < \sqrt{2}$, then $X$ must be a complex affine linear subspace.
\end{theorem}

Theorem~\ref{thm:tang_cone_union_linear_spaces} can be viewed as a quantitative strengthening of 
Proposition~\ref{prop:1-lne-rigidity} for this class of sets: the conclusion 
persists not only at $dist(X) = 1$ but for the entire range $dist(X) < \sqrt{2}$. As 
we shall see in Section~\ref{sec:remarks}, the bound $dist(X) < \sqrt{2}$ is sharp.

The key ingredient is the following lemma, which extends the computation in Example~\ref{ex:two-lines-sharp} to arbitrary finite unions of linear subspaces.

\begin{lemma}\label{lem:union-subspaces}
Let $L_1,\dots,L_r\subset\C^n$, $r\geq 2$, be complex linear subspaces such that $L_i\not\subset L_j$ for all $i\neq j$, and let $Z=L_1\cup\dots\cup L_r$. Then $dist(Z)\geq\sqrt{2}$.
\end{lemma}

\begin{proof}
Among the subspaces $L_i\cap L_j$ with $i\neq j$, choose one, say $M=L_a\cap L_b$, of maximal dimension, and let $I=\{i \mathrel{;} M\subset L_i\}$; note that $a,b\in I$. If $i,j\in I$ are distinct, then $L_i\cap L_j\supset M$, so $L_i\cap L_j=M$ by the maximality of $\dim M$.

Let $Q\colon\C^n\to M^{\perp}$ be the orthogonal projection and, for $i\in I$, set $L_i'=L_i\cap M^{\perp}$. Since $M\subset L_i$, we have $L_i=M\oplus L_i'$ and $Q(L_i)=L_i'$. For distinct $i,j\in I$, we get $L_i'\cap L_j'=M\cap M^{\perp}=\{0\}$. Moreover, $L_a'\neq\{0\}$, since otherwise $L_a=M\subset L_b$; similarly, $L_b'\neq\{0\}$.

For $i\notin I$, the subspace $L_i\cap M$ is a proper subspace of $M$. Since a complex vector space is not a finite union of proper subspaces, there is $p\in M$ with $p\notin L_i$ for all $i\notin I$ (if $M=\{0\}$, then $I=\{1,\dots,r\}$ and $p=0$). Let $\rho>0$ be smaller than $\mathrm{d}(p,L_i)$ for all $i\notin I$ (any $\rho>0$ if there is no such $i$). Since $p\in L_i$ for $i\in I$, we have
\[
Z\cap B(p,\rho)=Z_I\cap B(p,\rho),\qquad\text{where } Z_I=\bigcup_{i\in I}L_i .
\]

Choose unit vectors $u\in L_a'$ and $v\in L_b'$. Replacing $v$ by $e^{i\varphi}v$ for a suitable $\varphi\in\R$, we may assume that $\langle u,v\rangle=|\langle u,v\rangle|\geq0$, so that $\|u-v\|^2=2-2|\langle u,v\rangle|\leq2$. For $0<s\leq\rho/2$, set $x=p+su$ and $y=p+sv$. Then $x\in L_a\subset Z$, $y\in L_b\subset Z$ and $\|x-y\|=s\|u-v\|\leq\sqrt{2}\,s$.

Let $\gamma$ be a rectifiable path in $Z$ from $x$ to $y$. If $\gamma$ leaves $B(p,\rho)$, then, since $\|x-p\|=\|y-p\|=s$, we get $\operatorname{Length}(\gamma)\geq2(\rho-s)\geq\rho\geq2s$. Otherwise, $\gamma$ is contained in $Z_I$, and $\beta=Q\circ\gamma$ is a rectifiable path in $Q(Z_I)=\bigcup_{i\in I}L_i'$ from $Q(x)=su$ to $Q(y)=sv$ (recall that $p\in M$ and $u,v\in M^{\perp}$). The set $\bigcup_{i\in I}L_i'\setminus\{0\}$ is the disjoint union of the finitely many closed subsets $L_i'\setminus\{0\}$, $i\in I$, each of which is therefore also open in it. Since $su\in L_a'\setminus\{0\}$ and $sv\in L_b'\setminus\{0\}$ with $a\neq b$, the path $\beta$ passes through the origin. As $Q$ is $1$-Lipschitz,
\[
\operatorname{Length}(\gamma)\geq\operatorname{Length}(\beta)\geq\|su\|+\|sv\|=2s .
\]
In both cases $\operatorname{Length}(\gamma)\geq2s$, hence $\din(x,y)\geq2s$ and
\[
dist(Z)\geq\frac{\din(x,y)}{\|x-y\|}\geq\frac{2s}{\sqrt{2}\,s}=\sqrt{2}. \qedhere
\]
\end{proof}

We are now ready to prove Theorem~\ref{thm:tang_cone_union_linear_spaces}. The strategy is to 
show that the tangent cone at infinity $C(X, \infty)$ is a linear subspace, and then invoke the Lipschitz rigidity result of \cite{fernandes-sampaio-jga}.

\begin{proof}[Proof of Theorem~\ref{thm:tang_cone_union_linear_spaces}]
	Let $X \subset \C^n$ be as in the statement, with $c:=dist(X) < \sqrt{2}$. In particular, $X$ is $c$-Lipschitz normally embedded.
	
	By Theorem~\ref{thm:c-lne-tangent-cone-infinity}, $C(X, \infty)$ is $c$-Lipschitz normally embedded, that is, $dist(C(X,\infty))\leq c<\sqrt{2}$. By hypothesis, $C(X,\infty)$ is a finite union of complex linear subspaces; discarding those contained in others, we may write $C(X,\infty)=L_1\cup\dots\cup L_r$ with $L_i\not\subset L_j$ for $i\neq j$. If $r\geq2$, Lemma~\ref{lem:union-subspaces} gives $dist(C(X,\infty))\geq\sqrt{2}$, a contradiction. Hence $r=1$, that is, $C(X,\infty)$ is a linear subspace of $\C^n$.
	
	Since $X$ is a complex algebraic set 
	which is LNE (at infinity) and $C(X, \infty)$ is 
	a linear subspace, it follows from 
	\cite[Theorem~1.2]{fernandes-sampaio-jga} that $X$ is an affine linear 
	subspace of $\C^n$. 
\end{proof}

\begin{corollary}\label{cor:algebraic_curves}
   Let $X \subset \C^n$ be a connected algebraic affine complex curve. If $dist(X) < \sqrt{2}$, then $X$ is a 
complex line.
\end{corollary}

\section{Sharpness and further remarks}\label{sec:remarks}

\subsection{Sharpness for the constant}
We show that the constant $\sqrt{2}$ in Theorem~\ref{thm:tang_cone_union_linear_spaces} is 
optimal by exhibiting two algebraic curves (one singular and one smooth) that are $\sqrt{2}$-LNE but are not complex lines.

\medskip

\begin{example}
The curve $X = \{(x, y) \in \C^2 \mathrel{;} xy = 0\}$ is the union of the 
two coordinate axes, which meet orthogonally at the origin. By 
Example~\ref{ex:two-lines-sharp} with $\theta = \pi/2$, the sharp $c$-LNE 
constant of $X$ is $c = \sqrt{2}$. Thus $X$ is a connected algebraic curve 
that is $\sqrt{2}$-LNE but is not a complex line.
\end{example}

\subsection{Higher-dimensional generalizations}

It is natural to ask whether an analogous rigidity threshold exists for 
higher-dimensional complex algebraic sets. For a $d$-dimensional complex 
algebraic set $X \subset \C^n$, the tangent cone at infinity $C(X, \infty)$ 
is a $d$-dimensional complex algebraic cone. The critical $c$-LNE constant 
below which $X$ must be an affine linear subspace may depend on $d$ and on the geometry of such cones. Determining the sharp threshold in dimension 
$d \geq 2$ remains an open problem.

\begin{proposition}
Let $X=\{(x,y,z)\in \C^3;x^2+y^2+z^2=0\}$. Then $dist(X)= \sqrt{2} \sin\left(\frac{\pi}{2\sqrt{2}}\right)<\sqrt{2}$. Moreover, the supremum is a maximum, attained e.g. at $p_0=\tfrac1{\sqrt2}(1,i,0)$ and $q_0=\overline{p_0}$.

\end{proposition}

\begin{proof}
\textbf{Step 1: the link and its inner metric.}
Write $w=u+iv$ with $u,v\in\mathbb{R}^3$. Since $\sum_j w_j^2=\|u\|^2-\|v\|^2+2i\langle u,v\rangle$,
$$X=\{(u,v):\ \|u\|=\|v\|,\ \langle u,v\rangle=0\},$$
a real cone over its link $L=X\cap\mathbb{S}^5=\{(u,v):\|u\|=\|v\|=\tfrac1{\sqrt2},\ u\perp v\}$, a compact smooth $3$-manifold.
The map $\Phi:SO(3)\to L$, $\Phi(R)=\tfrac1{\sqrt2}(Re_1,Re_2)$, is a diffeomorphism, with inverse
$(u,v)\mapsto(\sqrt2u\,|\,\sqrt2v\,|\,2\,u\times v)$ (matrix with these columns).

For $a\in\mathbb{R}^3$, let $\hat a\in\mathfrak{so}(3)$ denote the skew-symmetric matrix defined by
$\hat a\,x=a\times x$ for all $x\in\mathbb{R}^3$; the map $a\mapsto\hat a$ is a linear isomorphism
$\mathbb{R}^3\to\mathfrak{so}(3)$. 
For a $C^1$ curve $R(t)$ in $SO(3)$, let $\omega(t)\in\mathbb{R}^3$ be the unique vector such that $\hat\omega(t)=R(t)^{-1}\dot R(t)$. Thus
$\dot R=R\hat\omega$. Then
$\tfrac{d}{dt}\Phi(R)=\tfrac1{\sqrt2}\big(R(\omega\times e_1),R(\omega\times e_2)\big)$, so that
\begin{equation}\label{eq:metric}
\Big\|\tfrac{d}{dt}\Phi(R)\Big\|^2
=\tfrac12\big(\|\omega\times e_1\|^2+\|\omega\times e_2\|^2\big)
=\tfrac12(\omega_1^2+\omega_2^2)+\omega_3^2=:Q(\omega).
\end{equation}
Let $d_L$ be the inner metric of $L$ induced by the Euclidean metric. Since $L\subset\mathbb{S}^5$, every curve
in $L$ from $p$ to $q$ is at least as long as the great-circle arc, hence
\begin{equation}\label{eq:angle}
d_L(p,q)\ \ge\ \angle(p,q)\in[0,\pi].
\end{equation}
Finally, $SO(3)$ acts on $\mathbb{C}^3$ by $A(u+iv)=Au+iAv$; this is a real-linear isometry of $\mathbb{R}^6$
preserving $X$ and $L$, hence an isometry of $(X,d_X)$ and of $(L,d_L)$, and $A\,\Phi(R)=\Phi(AR)$ shows that the
action on $L$ is transitive.

\textbf{Step 2: the cone formula.}
Since $X$ is a cone over $L$, $(X,d_X)$ is the Euclidean cone over $(L,d_L)$ in the sense of
\cite[\S3.6.2]{BBI}:
\begin{equation}\label{eq:cone}
d_X(rp,sq)^2=r^2+s^2-2rs\cos\big(\min\{d_L(p,q),\pi\}\big),\qquad r,s\ge0,\ p,q\in L.
\end{equation}
Indeed, in polar coordinates $w=r\ell$ ($r>0$, $\ell\in L$), one has $\langle\ell,\dot\ell\rangle=0$, hence
$\|\dot w\|^2=\dot r^2+r^2\|\dot\ell\|^2$ for every Lipschitz curve in $X\setminus\{0\}$; thus the Euclidean length of
a curve in $X\setminus\{0\}$ equals its length in the cone over $(L,d_L)$, while a curve through $0$ has length
$\ge r+s$. In particular, for $p,q\in L$,
$d_X(p,q)=2\sin\big(\alpha(p,q)/2\big)$ with $\alpha(p,q):=\min\{d_L(p,q),\pi\}$.

\textbf{Step 3: reduction to pairs on the link.}
Let $p,q\in L$, $\theta=\angle(p,q)$, $\alpha=\alpha(p,q)$; by \eqref{eq:angle}, $\alpha\ge\theta$.
For $r,s>0$, by \eqref{eq:cone},
$$\frac{d_X(rp,sq)^2}{\|rp-sq\|^2}
=\frac{r^2+s^2-2rs\cos\alpha}{r^2+s^2-2rs\cos\theta}
=\frac{1-\mu\cos\alpha}{1-\mu\cos\theta},\qquad \mu=\frac{2rs}{r^2+s^2}\in(0,1].$$
The derivative of the right-hand side with respect to $\mu$ is $(\cos\theta-\cos\alpha)/(1-\mu\cos\theta)^2\ge0$,
so the ratio is maximal at $\mu=1$, i.e.\ $r=s$, and by homogeneity at $r=s=1$. Pairs with $p=q$ or with one point
equal to $0$ have ratio $1$. Using the transitivity of Step~1 and fixing $p_0=\Phi(I)=\tfrac1{\sqrt2}(e_1,e_2)$,
\begin{equation}\label{eq:red}
\operatorname{dist}(X)=\sup_{q\in L\setminus\{p_0\}}\frac{2\sin\big(\alpha(p_0,q)/2\big)}{\|p_0-q\|}.
\end{equation}

\textbf{Step 4: the outer distance.}
Let $q=\Phi(R)$ with $R\ne I$, and let $n\in\mathbb{S}^2$, $\theta\in(0,\pi]$ be the axis and angle of $R$, so that by
Rodrigues' formula $R=\cos\theta\,I+\sin\theta\,\hat n+(1-\cos\theta)\,nn^{T}$ (see e.g.\ \cite[\S2.2, Formula 2.14]{MLS}).
Since $\|e_j-Re_j\|^2=2-2R_{jj}$,
\begin{eqnarray*}
    \|p_0-q\|^2&=& \tfrac12\sum_{j=1}^2(2-2R_{jj})\\
               &=& 2-\operatorname{tr}R+R_{33}\\
               &=& 2-(1+2\cos\theta)+\cos\theta+(1-\cos\theta)n_3^2\\
               &=& (1-\cos\theta)(1+n_3^2).
\end{eqnarray*}
Hence, with
\begin{equation}\label{eq:s}
s:=\sqrt{\tfrac{1+n_3^2}{2}}\in\big[\tfrac1{\sqrt2},1\big],\qquad
\|p_0-q\|=2s\sin(\theta/2).
\end{equation}

\textbf{Step 5: upper bound for the inner distance.}
The curve $\gamma(t)=\Phi(\exp(t\hat n))$, $0\le t\le\theta$, joins $p_0$ to $q$ in $L$ and has constant body
velocity $\omega\equiv n$, because $\tfrac{d}{dt}\exp(t\hat n)=\exp(t\hat n)\,\hat n$. By \eqref{eq:metric} its speed
is $\sqrt{Q(n)}=\sqrt{\tfrac12(n_1^2+n_2^2)+n_3^2}=s$. Therefore
$$d_L(p_0,q)\le s\theta\le\pi ,$$
so $\alpha(p_0,q)=d_L(p_0,q)\le s\theta$, and since $y\mapsto 2\sin(y/2)$ is increasing on $[0,\pi]$,
\begin{equation}\label{eq:upper}
d_X(p_0,q)\le 2\sin(s\theta/2).
\end{equation}

\textbf{Step 6: the upper bound for $\operatorname{dist}(X)$.}
Put $x=\theta/2\in(0,\pi/2]$. By \eqref{eq:s} and \eqref{eq:upper},
$$\frac{d_X(p_0,q)}{\|p_0-q\|}\le\frac{\sin(sx)}{s\sin x}=\frac{x}{\sin x}\cdot\frac{\sin(sx)}{sx}.$$
The function $y\mapsto\sin y/y$ is decreasing on $(0,\pi]$ (its derivative $(y\cos y-\sin y)/y^2$ is negative,
as $\tan y>y$ on $(0,\pi/2)$ and $y\cos y\le0<\sin y$ on $[\pi/2,\pi)$). Since $sx\in(0,\pi/2]$ and
$s\ge1/\sqrt2$, the right-hand side is largest at $s=1/\sqrt2$, i.e.
$$\frac{d_X(p_0,q)}{\|p_0-q\|}\le f(x):=\sqrt2\,\frac{\sin(x/\sqrt2)}{\sin x}.$$
We claim $f$ is increasing on $(0,\pi/2]$. With $c=1/\sqrt2$ one has $f'(x)=\sqrt2\,h(x)/\sin^2x$, where
$h(x)=c\cos(cx)\sin x-\sin(cx)\cos x$ satisfies $h(0)=0$ and
$h'(x)=(1-c^2)\sin(cx)\sin x>0$ on $(0,\pi)$; hence $h>0$ on $(0,\pi/2]$ and the claim follows.
Consequently the ratio is at most $f(\pi/2)=\sqrt2\sin\big(\tfrac{\pi}{2\sqrt2}\big)$ for every $q\ne p_0$,
and by \eqref{eq:red}
$$\operatorname{dist}(X)\le\sqrt2\,\sin\!\Big(\frac{\pi}{2\sqrt2}\Big).$$

\textbf{Step 7: the bound is attained.}
Let $q_0=\Phi(\operatorname{diag}(1,-1,-1))=\tfrac1{\sqrt2}(e_1,-e_2)$, i.e.\ $n=e_1$, $\theta=\pi$, $s=1/\sqrt2$;
in complex coordinates $p_0=\tfrac1{\sqrt2}(1,i,0)$ and $q_0=\overline{p_0}$. By \eqref{eq:s},
$\|p_0-q_0\|=\sqrt2$, and by Step~5, $d_L(p_0,q_0)\le\pi/\sqrt2$. For the reverse inequality, the projection
$\pi_v:L\to\mathbb{S}^2(1/\sqrt2)$, $(u,v)\mapsto v$, does not increase the Euclidean length of curves, since
$\|\dot v\|\le\|(\dot u,\dot v)\|$; therefore
$$d_L(p_0,q_0)\ \ge\ d_{\mathbb{S}^2(1/\sqrt2)}\Big(\tfrac{e_2}{\sqrt2},-\tfrac{e_2}{\sqrt2}\Big)=\frac{\pi}{\sqrt2},$$
the two points being antipodal on a round sphere of radius $1/\sqrt2$, whose inner metric is the arc-length metric.
Thus $d_L(p_0,q_0)=\pi/\sqrt2<\pi$, so by Step~2, $d_X(p_0,q_0)=2\sin\big(\tfrac{\pi}{2\sqrt2}\big)$ and
$$\frac{d_X(p_0,q_0)}{\|p_0-q_0\|}=\sqrt2\,\sin\!\Big(\frac{\pi}{2\sqrt2}\Big).$$
Together with Step~6 this proves $\operatorname{dist}(X)=\sqrt2\sin\big(\tfrac{\pi}{2\sqrt2}\big)$; the inequality
$<\sqrt2$ is immediate since $\tfrac{\pi}{2\sqrt2}\notin\tfrac{\pi}{2}+2\pi\mathbb{Z}$.
\end{proof}

\begin{remark}
The quadratic form $Q$ in \eqref{eq:metric} is a left-invariant metric of Berger type on $SO(3)\cong L$.
The direction $\omega=e_3$ is the orbit of the circle action $w\mapsto e^{i\vartheta}w$ (the Hopf fibre of
$\mathbb{S}^5$ restricted to $L$), and it is \emph{longer} than the two transverse directions ($Q(e_3)=1$ versus
$Q(e_1)=Q(e_2)=\tfrac12$). The extremal pair $(p_0,\overline{p_0})$ is joined in $L$ by the half-turn about $e_1$,
a transverse direction, which is why the maximal distortion is governed by the factor $1/\sqrt2$.
\end{remark}

\subsection{Open question}

We conclude with the following question, motivated by this work.

\begin{enumerate}
\item[-] What is the sharp $c$-LNE rigidity threshold for complex algebraic 
  surfaces? That is, what is the largest $c(2)$ such that every connected 
  $c$-LNE complex algebraic surface with $c < c(2)$ must be an affine 
  linear subspace?
\end{enumerate}

\bigskip

\noindent {\bf Acknowledgments.} During the preparation of this work, the authors utilized AI assistants (Gemini and Claude) to improve language and readability, organize ideas and proofs. All AI-generated suggestions were independently verified by the authors, who take full responsibility for the content and accuracy of the final manuscript.


\end{document}